\documentclass[12pt]{article}
\usepackage{theorem,amsmath,amssymb,latexsym}
\usepackage[all]{xy}

\usepackage{mathrsfs,stmaryrd}

\theoremstyle{plain}
\theoremheaderfont{\normalfont\bfseries}
\theorembodyfont{\itshape}

\newtheorem{theorem}{Theorem}[section]
\newtheorem{proposition}[theorem]{Proposition}
\newtheorem{lemma}[theorem]{Lemma}
\newtheorem{corollary}[theorem]{Corollary}

\theorembodyfont{\upshape}
{\theoremstyle{break}

}

\theorembodyfont{\upshape}
\newtheorem{remark}[theorem]{Remark}

\theorembodyfont{\upshape}

\theorembodyfont{\upshape}

\theorembodyfont{\upshape}

\theorembodyfont{\itshape}

\theorembodyfont{\upshape}

\theorembodyfont{\upshape}

\theorembodyfont{\upshape}

\theorembodyfont{\upshape}

\theorembodyfont{\upshape}
\newtheorem{example}{Example}

\theorembodyfont{\upshape}
\newtheorem{proof}{Proof}

\theorembodyfont{\upshape}

\newcommand{\myemail}[1]{\indent \emph{E-mail:} {\tt #1}}
\newcommand{\myaddress}[1]{\indent {\sc #1}\par}

\title{On the $\eta$-inverted sphere}
\author{Oliver R\"ondigs}
\date{January 8, 2018}

\begin{document}
\maketitle

\begin{abstract}
The first and second homotopy groups
of the $\eta$-inverted sphere spectrum over a field
of characteristic not two are zero.
A cell presentation of higher Witt theory is given as well, at
least over the complex numbers.
\end{abstract}
\section{Introduction}
\label{section:introduction}

Let $\eta\colon \mathbf{A}^2\smallsetminus\{0\}\to \mathbf{P}^1$ denote the first
algebraic Hopf map given by sending a nonzero pair $(x,y)$ to the line it generates in the plane. It induces an element of the same name
$\eta\in \pi_{1,1}\mathbf{1}$ in the motivic stable
homotopy groups of spheres. Hence $\eta$ acts on the
motivic stable homotopy groups of the sphere spectrum, in fact
of any motivic spectrum. Inverting this
action produces trivial 
groups in degrees one and two for the sphere spectrum.

\begin{theorem}\label{thm:introduction}
  Let $F$ be a field of characteristic not two. Then
  \[ \pi_{n+1,n}\mathbf{1}[\tfrac{1}{\eta}] = \pi_{n+2,n}\mathbf{1}[\tfrac{1}{\eta}] = 0\] for every integer $n$.
\end{theorem}

This vanishing result enters the computation of the first line
\[ \bigoplus_{n\in \mathbb{Z}} \pi_{n+1,n} \mathbf{1} \]
of motivic stable homotopy groups of spheres given in \cite{rso.sphere}.
Fabien Morel
identified the zero line of motivic stable homotopy groups of
spheres, that is, the direct sum
\[ \bigoplus_{n\in \mathbb{Z}} \pi_{n,n} \mathbf{1} \]
with the graded Milnor-Witt $K$-theory
of the base field -- see \cite[Theorem 6.2.1]{morel.pi0},
and the corresponding unstable
statement \cite[Theorem 4.9]{MorelICM2006}, documented in detail
for perfect infinite fields in \cite{morel.field}. This identification
implies that $\pi_{n,n}\mathbf{1}[\tfrac{1}{\eta}]$ is isomorphic
to the Witt ring of the base field for every integer $n$.
Impressive work of Guillou and Isaksen on the motivic Adams spectral
sequence at the prime two computes 
$\pi_{n+k,n}\mathbf{1}^\wedge_2[\tfrac{1}{\eta}]$ for all integers $n$ and
$k$ over the complex and real numbers \cite{guillou-isaksen}, 
\cite{guillou-isaksen.real}. 
It implies that $\pi_{n+3,n}\mathbf{1}^\wedge_2$ is
nontrivial, at least over the real and the complex numbers.
The next task
is to determine
$\pi_{n+3,n}\mathbf{1}[\tfrac{1}{\eta}]$ for an arbitrary field of
characteristic not two, perhaps via the 
resolution 
of $\mathbf{1}[\tfrac{1}{\eta}]$ via connective Witt theory
$\mathsf{cKW}$
employed in Section~\ref{sec:an-adams-resolution}. Only the
first two maps
\begin{equation}\label{eq:3} 
  \mathbf{1}[\tfrac{1}{\eta}] \to \mathsf{cKW} \to \mathsf{cKW}\wedge
  \mathsf{cKW} 
\end{equation}
of this resolution
enter the proof of Theorem~\ref{thm:introduction}. The zero line computation
of Morel cited above allows to deduce the vanishing 
in Theorem~\ref{thm:introduction} from the connectivity of the second map
in~(\ref{eq:3}), as Section~\ref{sec:an-adams-resolution} explains.
This connectivity is obtained in Section~\ref{sec:oper-high-witt}, at
least after inverting 2, by transferring results from topology
on connective real $K$-theory to connective Witt theory. Proving
that this transfer works proceeds by a real version of Joseph Ayoub's 
model for the Betti realization, to be explained in 
Section~\ref{sec:real-betti} and exploited in Section~\ref{sec:witt-theory}.
The argument sketched so far would provide Theorem~\ref{thm:introduction}
only after inverting 2\footnote{See
\cite[Proposition 36]{bachmann.real} for a different proof of
this vanishing statement after inverting 2.}, 
which annihilates Witt theory for fields which are
not formally real -- see Section~\ref{sec:fields-odd-char}. 
The integral computation requires knowledge
on 2-complete computations, supplied by results 
from \cite{rso.sphere}
with the help of the convergence result \cite[Theorem 1]{hu-kriz-ormsby}.
The passage from 2-inverted to integral results is explained 
in Section~\ref{sec:vanishing}. 

Theorem~\ref{thm:introduction} implies that a cell
presentation of connective Witt theory starts by attaching
4-cells via the generators 
\[ \Sigma^{3,0}\mathbf{1}[\tfrac{1}{\eta}] \to \mathbf{1}[\tfrac{1}{\eta}]\]
of $\pi_{3,0}\mathbf{1}[\tfrac{1}{\eta}]$. For the complex numbers,
where complete information is available by \cite{guillou-isaksen},
a very small and complete 
cell presentation of connective Witt theory is obtained
at the end of Section~\ref{sec:witt-theory} (see also \cite{hornbostel.nil}). 
This cell presentation
is not used in the proof of Theorem~\ref{thm:introduction}, but
included for comparison with
the impressive rational result \cite[Corollary 2.8]{alp}. Cell presentations and
connectivity are discussed in Section~\ref{sec:cell-presentations}.
The following notation is used throughout.
\vspace{0.05in}

\begin{tabular}{l|l}
$F$, $S$ & field, finite dimensional separated Noetherian base scheme \\
$\mathbf{Sm}_{S}$ & smooth schemes of finite type over $S$ \\
$\mathbf{SH}(S)$  & the motivic stable homotopy category of $S$\\ 
$\mathsf{E}$, $\mathbf{1}=S^{0,0}$ & generic motivic spectrum, the motivic sphere spectrum  \\
$S^{s,t}$, $\Sigma^{s,t}$ & motivic $(s,t)$-sphere, $(s,t)$-suspension  \\
$\pi_{s,t}\mathsf{E} = [S^{s,t},\mathsf{E}]$ & motivic stable homotopy groups of $\mathsf{E}$\\
$\underline{\pi}_{s,t}\mathsf{E}$ & sheaves of motivic stable homotopy groups of $\mathsf{E}$\\
$\mathsf{KQ}$, $\mathsf{KW}$ & hermitian $K$-theory, Witt-theory 
\end{tabular}
\vspace{0.05in}

\noindent
The suspension convention 
is such that $\mathbf{P}^1\simeq S^{2,1}$ and 
$\mathbf{A}^{1}\smallsetminus \{0\}\simeq S^{1,1}$.

\bigskip

A previous version of this work was included in a preliminary
version of \cite{rso.sphere}, but separated for the sake of
presentation. Theorem~\ref{thm:introduction} enters 
only \cite[Section 5]{rso.sphere}, and its proof requires
results from \cite[Sections 2--4]{rso.sphere} -- see Section~\ref{sec:vanishing} for details. I thank the RCN program
{\em Topology in Norway\/}, and in particular the University of Oslo 
and Paul Arne {\O}stv{\ae}r, as well as
the DFG priority program {\em Homotopy theory and algebraic geometry\/}
for support. I thank Joseph Ayoub
for his input regarding a real version of his
model of complex Betti realization. Tom Bachmann, Jeremiah Heller, and an
anonymous referee
receive my thanks for helpful comments.

\section{Cellularity and connectivity}
\label{sec:cell-presentations}

The following definitions regarding cells
are modelled on \cite{dugger-isaksen.cell}
and \cite{hu-kriz-ormsby}.
{\em Attaching a cell\/} to a motivic spectrum $\mathsf{E}$
refers to the process of forming the pushout of
\begin{equation}\label{eq:cell} 
D^{s+1,t}\mathbf{1} \hookleftarrow \Sigma^{s,t} \mathbf{1} \xrightarrow{f} \mathsf{E} 
\end{equation}
for some
map $f$ in the category of motivic spectra. The arrow pointing
to the left in~(\ref{eq:cell}) denotes the canonical
inclusion into
$D^{s+1,t}\mathbf{1}$, the simplicial mapping cylinder of
the map $\Sigma^{s,t}\mathbf{1} \to \ast$. The pushout $\mathsf{D}$ then consists of
$\mathsf{E}$, together with a cell of 
{\em dimension\/} $(s+1,t)$ and {\em weight\/} $t$. 
More generally, one may attach a collection of cells
indexed by some set $I$ by forming the pushout 
of a diagram
\[ \bigvee_{i\in I}D^{s_i,t_i} \mathbf{1}
\hookleftarrow
\bigvee_{i\in I}\Sigma^{s_i,t_i} \mathbf{1} \xrightarrow{f} \mathsf{E} \]
in the category of motivic spectra.
A {\em cell presentation\/} of a map $f\colon \mathsf{D}\to \mathsf{E}$ of motivic spectra 
consists of 
a sequence of motivic spectra
\begin{equation}\label{eq:cell-presentation} \mathsf{D}=\mathsf{D}_{-1}\xrightarrow{d_0} \mathsf{D}_0\xrightarrow{\sim} \mathsf{D}^{\prime}_0\xrightarrow{d_1} \mathsf{D}_1\xrightarrow{\sim} \mathsf{D}^{\prime}_1\to \dotsm \xrightarrow{d_{n}} \mathsf{D}_{n} 
\xrightarrow{\sim} \mathsf{D}^{\prime}_n\xrightarrow{d_{n+1}} \dotsm 
\end{equation}
with canonical map to the colimit $c\colon \mathsf{D}\to \mathsf{D}_\infty$, attaching maps
\[ \bigvee_{i\in I_n}\Sigma^{s_i,t_i} \mathbf{1} \xrightarrow{\alpha_n} 
\mathsf{D}^{\prime}_{n-1} \]
for every natural number $n$, such that
$\mathsf{D}_{n}$ is obtained by attaching cells
to $\mathsf{D}^{\prime}_{n-1}$ along $\alpha_n$, and
a weak equivalence $w\colon \mathsf{D}_\infty\xrightarrow{\sim} \mathsf{E}$ with
$w\circ c =f$. 
Furthermore, all arrows in diagram~(\ref{eq:cell-presentation})
labelled with $\sim$ denote acyclic cofibrations.

\begin{remark}\label{rem:relative-cell-pres}
  The most important instance is the 
  absolute case, that is, a cell presentation
  of a map $\ast\to \mathsf{E}$. Then one speaks of
  a cell presentation of $\mathsf{E}$. 
  The distinction between a motivic spectrum
  $\mathsf{E}$ and a cell presentation $\mathsf{D}_\infty$ of it might be 
  neglected if no confusion can arise.
\end{remark}

\begin{example}\label{ex:bounded}
  The suspension spectrum $\Sigma^\infty \mathbf{P}^\infty$
  admits a cell presentation, having precisely one cell of
  dimension $(2n,n)$ for every natural number $n$.
\end{example}

Let $\mathbf{SH}(S)^{\mathrm{cell}}$ denote the full subcategory of $\mathbf{SH}(S)$ 
of motivic spectra
admitting a cell presentation. It can be identified with
the homotopy category of {\em cellular\/} motivic spectra, that is, the
smallest full localizing subcategory of $\mathbf{SH}(S)$
containing the spheres $S^{s,t}$, as \cite[Remark 7.4]{dugger-isaksen.cell}
and the following statement
show. 

\begin{lemma}\label{lem:cellular-presentation}
  Let $\mathsf{E}$ be a motivic spectrum. If it admits a
  cell presentation, it is cellular. Conversely,
  if it is cellular, it admits a cell presentation.
\end{lemma}

\begin{proof}
  The first statement follows from the definitions.
  For the second statement, let $\mathsf{E}$ be a cellular
  motivic spectrum which is fibrant. One constructs
  inductively a suitable sequence of motivic spectra.
  Start with
  \[ \mathsf{D}_0:= \bigvee_{\alpha\in \pi_{s,t}\mathsf{E},s,t} \Sigma^{s,t} \mathbf{1} 
  \xrightarrow{\vee\alpha} \mathsf{E} \]
  and factor the canonical map $\mathsf{D}_0\to \mathsf{E}$ 
  as an acyclic cofibration $\mathsf{D}_0\xrightarrow{\sim} \mathsf{D}_0^\prime$
  followed by a fibration $\mathsf{D}_0^\prime \to \mathsf{E}$.
  The latter induces a surjection
  on $\pi_{s,t}$ for all $s,t$ by construction. For every $s,t$, choose 
  lifts of generators of the
  kernel of $\pi_{s,t}(\mathsf{D}_0^\prime\to \mathsf{E})$ and use these to attach cells
  to  $\mathsf{D}_0^\prime$, leading to a map 
  $\mathsf{D}_1\to \mathsf{E}$. As before, factor it as an acyclic cofibration,
  followed by a fibration $\mathsf{D}_1^\prime \to \mathsf{E}$.
  This fibration still
  induces a surjection on $\pi_{s,t}$ for all $s,t$. Iterating
  this procedure leads to a map
  \[ \mathsf{D}_\infty = \operatornamewithlimits{colim}_n{\mathsf{D}_n}  \to \mathsf{E} \]
  which induces a surjective and injective map on $\pi_{s,t}$
  for all $s,t$. The statement on injectivity requires that
  $\Sigma^{s,t}\mathbf{1}$ is compact, whence any element in the kernel
  lifts to a finite stage and hence is killed in the next stage.
  Since both $\mathsf{D}_\infty$ and $\mathsf{E}$ are cellular, the map
  $\mathsf{D}_\infty\to\mathsf{E}$ is even a weak 
  equivalence \cite[Cor.~7.2]{dugger-isaksen.cell}.
\end{proof}

The motivic stable homotopy category $\mathbf{SH}(S)$ is
equipped with the {\em homotopy t-structure\/} 
\[ \mathbf{SH}(S)_{\geq n} \hookrightarrow \mathbf{SH}(S) \hookleftarrow \mathbf{SH}(S)_{\leq n} \]
where $\mathbf{SH}(S)_{\geq n} = \langle \Sigma^{s,t}\Sigma^\infty X_+ \vert X\in \mathbf{Sm}_S,
s-t\geq n\rangle$ is the full subcategory 
generated under homotopy
colimits and extensions by the shifted motivic suspension spectra 
of smooth $S$-schemes of connectivity at least $n$. See 
\cite[Section 2.1]{hoyois} for details, and in particular
\cite[Theorem 2.3]{hoyois} for the identification with
Morel's original definition via Nisnevich sheaves of
motivic stable homotopy groups in case $S$ is the spectrum of a field.

\section{Real realization}\label{sec:real-betti}

Real Betti realization
will be employed in order to give a topological
interpretation of connective Witt theory over the real numbers. 
It is defined as the homotopy-colimit preserving functor
$\mathbf{SH}(\mathbb{R})\to \mathbf{SH}$ from the motivic stable homotopy category
of the real numbers to the classical motivic stable
homotopy category which is determined by sending the
suspension spectrum of a smooth variety $X$ over $\mathbb{R}$
to the suspension spectrum of the topological space of
real points $X(\mathbb{R})$, equipped with the real analytic topology.
Another viewpoint on the real Betti realization is given by
equivariant stable homotopy theory with respect to the absolute
Galois group $C_2$ of $\mathbb{R}$: 
The real Betti realization is the composition
of geometric fixed points and the complex Betti realization; 
see \cite[Section 4.4]{ho}
for details. Real Betti realization is given
already unstably on the level of presheaves on $\mathbf{Sm}_\mathbb{R}$ with
values in simplicial sets, and as such it is a simplicial functor. 
There is an additive variant
for presheaves with values in simplicial abelian groups, or, equivalently, 
complexes of abelian groups.

\begin{theorem}\label{thm:real-points}
  Real Betti realization
  $\mathbb{R}^\ast\colon \mathbf{SH}(\mathbb{R})\to \mathbf{SH}$
  is strict symmetric monoidal and
  has a right adjoint $\mathbb{R}_\ast$. Moreover, the canonical
  map 
  \begin{equation}\label{eq:real-points}
    \mathsf{D}\wedge \mathbb{R}_\ast(\mathsf{E}) \to \mathbb{R}_\ast \bigl(\mathbb{R}^\ast(\mathsf{D})\wedge \mathsf{E}\bigr)
  \end{equation}
  is an equivalence for all motivic spectra 
  $\mathsf{D}\in \mathbf{SH}(\mathbb{R})$.
\end{theorem}

\begin{proof}
  The first statement 
  follows from \cite[Section 3.3]{mv}, and also from
  \cite[Proposition 4.8]{ho}. The second statement follows
  from \cite[Proposition 3.2]{fhm} for strongly dualizable
  motivic spectra. Since $\mathbf{SH}(\mathbb{R})$ is
  generated by strongly dualizable motivic spectra,
  the result follows.
\end{proof}

Identifying the real Betti realization 
of motivic spectra which are not suspension
spectra of smooth varieties is not immediate. However, Joseph Ayoub's
beautiful model for the complex Betti realization 
given in \cite[Th\'eor\`eme 2.67]{ayoub.galois-1} 
translates to a model for the real Betti realization. In order
to define this model, let $\mathcal{I}^n$ denote the pro-real analytic
manifold obtained by open neighborhoods of the compact unit
cube $[0,1]^n\subset \mathbb{R}^n$. Letting $n$ vary defines, in the standard
way, a cocubical pro-real analytic manifold $\mathcal{I}^\bullet$ as
defined in \cite[Def.~A.1]{ayoub.galois-1}. One observes that
$\mathcal{I}^\bullet$ is in fact a pseudo-monoidal $\Sigma$-enriched
cocubical pro-real analytic manifold 
\cite[Def.~A.29]{ayoub.galois-1}. 
Let $\mathcal{R}_n$ denote the $\mathbb{R}[t_1,\dotsc,t_n]$-algebra
of real analytic functions defined on an open 
neighborhood of $[0,1]^n$. 

\begin{proposition}[Ayoub]\label{prop:ayoub-popescu}
  The $\mathbb{R}[t_1,\dotsc,t_n]$-algebra
  $\mathcal{R}_n$ is Noetherian and regular.
\end{proposition}

\begin{proof}
  The $\mathbb{C}[t_1,\dotsc,t_n]$-algebra $\mathcal{R}_n\otimes_\mathbb{R} \mathbb{C}$
  is the algebra of complex analytic functions defined on an open
  neighborhood of $[0,1]^n$, which is Noetherian and regular.
  The required statement follows by Galois descent.
\end{proof}

Popescu's theorem \cite[Theorem 10.1]{spivakovsky.smoothing}
then implies that 
$\mathcal{R}_n$ is a filtered colimit of smooth 
$\mathbb{R}[t_1,\dotsc,t_n]$-algebras, and thus can be considered
as an affine pro-smooth $\mathbb{R}$-variety $\operatorname{Spec}(\mathcal{R}_n)$.
Hence any presheaf $K$ on $\mathbf{Sm}_\mathbb{R}$ admits a value
$K(\operatorname{Spec}(\mathcal{R}_n))$. Letting $n$ vary defines
a cubical object $K(\operatorname{Spec}(\mathcal{R}_\bullet))$.

\begin{theorem}[Ayoub]\label{thm:ayoub-real-betti}
  Let $C\colon \mathbf{Sm}_\mathbb{R}\to \mathbf{Cx}$ 
  be a presheaf of complexes of abelian groups.
  The real Betti realization of $C$ is quasi-isomorphic to
  the total complex $C(\operatorname{Spec}(\mathcal{R}_\bullet))$.
\end{theorem}

In fact, as in the complex case given in 
\cite[Th\'eor\`eme 2.61]{ayoub.galois-1}
it is possible to replace $\operatorname{Spec}(\mathcal{R}_n)$ in
Theorem~\ref{thm:ayoub-real-betti} by 
$\operatorname{Spec}(\mathcal{R}^{et}_n)$ where $\mathcal{R}^{et}_n$
is, roughly speaking, the largest sub-algebra of $\mathcal{R}_n$
which is pro-\'etale over $\mathbb{R}[t_1,\dots, t_n]$.

\section{Witt theory}
\label{sec:witt-theory}

Let $\mathsf{KQ}$ denote the motivic spectrum
for hermitian $K$-theory of quadratic forms \cite{hornbostel.hermitian}.
It is defined over $\operatorname{Spec}(\mathbb{Z}[\tfrac{1}{2}])$, and
hence over any scheme in which $2$ is invertible.
Note the identification
\[  \mathsf{KQ}_{s,t} = \pi_{s,t} \mathsf{KQ} =  \mathbb{G}W_{s-2t}^{[-t]}(F)\]
with Schlichting's higher Grothendieck-Witt groups \cite[Definition 9.1]{schlichting} for any field of characteristic different from two. 
The periodicity element is denoted
$\alpha\colon S^{8,4}\to \mathsf{KQ}$.
It is denoted $\beta$ in \cite{ananyevskiy.coop}.

\begin{theorem}
\label{theorem:KQcellular}
Suppose $S$ is a scheme over $\operatorname{Spec}(\mathbb{Z}[\tfrac{1}{2}])$. 
Then $\mathsf{KQ}$ is a cellular commutative motivic ring 
spectrum in $\mathbf{SH}(S)$ which is
preserved under base change.
\end{theorem}

\begin{proof}
  A ring structure is provided in~\cite{paninwalterBO}.
  Cellularity follows basically from the model given 
  in~\cite{paninwalterBO}; details may be found in \cite{rso.kqcell}.
  The statement regarding base change refers to the fact that, given
  a morphism $f\colon S\rightarrow S^\prime$ of schemes over 
  $\operatorname{Spec}(\mathbb{Z}[\tfrac{1}{2}])$, there is a canonical
  identification $f^\ast(\mathsf{KQ})\rightarrow
  \mathsf{KQ}$ of motivic spectra over $S$. It is induced by the 
  corresponding canonical identification of Grassmannians which serve
  to model $\mathsf{KQ}$ by~\cite{paninwalterBO}.
\end{proof}

Let $\mathsf{KW}$ be the motivic ring spectrum representing higher Balmer-Witt 
groups in the motivic stable homotopy category. It can be
described as $\mathsf{KW} = \mathsf{KQ}[\tfrac{1}{\eta}]$, where
$\eta\colon S^{1,1}\to
S^{0,0}$ is the first algebraic Hopf map.

\begin{corollary}\label{cor:kt-cellular}
Suppose $S$ is a scheme over $\operatorname{Spec}(\mathbb{Z}[\tfrac{1}{2}])$. 
Then $\mathsf{KW}$ is a cellular motivic ring 
spectrum in $\mathbf{SH}(S)$ which is
preserved under base change.
It is a commutative $\mathbf{1}[\tfrac{1}{\eta}]$-algebra.
\end{corollary}

Over fields, one may describe the coefficients of $\mathsf{KW}$ as follows:
\begin{equation}\label{eq:kt-coeff}
  \mathsf{KW}_{\ast,\ast}\cong \mathsf{KW}_{0,0}[\eta,\eta^{-1},\alpha,\alpha^{-1}]
  \cong W[\eta,\eta^{-1},\alpha,\alpha^{-1}]
\end{equation}
Here $W$ is the Witt ring of the base field
and $\alpha$ is the following composition:
\[ S^{8,4}\xrightarrow{\alpha}\mathsf{KQ}\xrightarrow{\mathrm{can.}}\mathsf{KW} \]
The remainder of the section will be devoted to specific
descriptions of the motivic spectrum $\mathsf{KW}$, after suitable
modifications, or over specific fields.
See \cite[Theorem 2.6.4]{scharlau} for the following
statement.

\begin{theorem}[Pfister]\label{thm:witt-2-torsion}
  Let $S$ be the spectrum of a field.
  The additive group underlying the graded ring
  $\mathsf{KW}_{\ast,\ast}[\tfrac{1}{2}]$ is torsion-free. 
\end{theorem}

In order to state the next theorem, let $\mathsf{KO}$ denote the topological
spectrum representing real topological $K$-theory. Its 0-connective
cover is denoted $\mathsf{ko}$. The coefficients of
the topological spectrum $\mathsf{KO}[\tfrac{1}{2}]$ are particularly simple:
A Laurent polynomial ring on a single generator $\alpha_\mathsf{top}$ of
degree 4.

\begin{theorem}[Brumfiel]\label{thm:real-points-kt}
  There is an equivalence $\mathbb{R}^\ast \mathsf{KW}[\tfrac{1}{2}] \simeq \mathsf{KO}[\tfrac{1}{2}]$
  of topological spectra. The
  unit map $\mathsf{KW}[\tfrac{1}{2}]\to \mathbb{R}_\ast\mathbb{R}^\ast \mathsf{KW}[\tfrac{1}{2}]$ is an equivalence
  in $\mathbf{SH}(\mathbb{R})$.
\end{theorem}

\begin{proof}
  This follows from \cite[Theorem 6.2]{ksw} 
  which provides a natural 
  equivalence of classical spectra
  \[ 
  \mathsf{KW}_0[\tfrac{1}{2}](X) 
  \to \mathsf{KO}[\tfrac{1}{2}](\mathbb{R}^\ast X)  \]
  for every  
  finite type $\mathbb{R}$-scheme $X$.
  The Yoneda lemma then implies that the map
  induced by the functor $\mathbb{R}^\ast$ on internal
  simplicial sets of morphisms is an equivalence
  for all smooth $\mathbb{R}$-schemes, whose
  $\Sigma^{s,t}$-suspensions generate 
  $\mathbf{SH}(\mathbb{R})$. Adjointness
  then implies that the unit map
  $\mathsf{KW}[\tfrac{1}{2}]\to \mathbb{R}_\ast\mathbb{R}^\ast \mathsf{KW}[\tfrac{1}{2}]$ is an equivalence
  in $\mathbf{SH}(\mathbb{R})$.
\end{proof}

Due to the $\alpha$-periodicity of its target, the unit $\mathbf{1} \to \mathsf{KW}$
is not $k$-connective for any $k\in \mathbb{Z}$. Let
$\psi\colon \mathsf{cKW}:= \mathsf{KW}_{\geq 0}\to \mathsf{KW}$ denote the 
$0$-connective cover with respect to the homotopy $t$-structure.
Since $\mathbf{1}$ is $0$-connective,
the $\mathbf{1}[\tfrac{1}{\eta}]$-algebra homomorphism
$u\colon \mathbf{1}[\tfrac{1}{\eta}]\to \mathsf{KW}$ 
factors uniquely as a $\mathbf{1}[\tfrac{1}{\eta}]$-algebra homomorphism
$cu \colon \mathbf{1}[\tfrac{1}{\eta}] \to \mathsf{cKW}$, even in a strict
sense if one uses \cite{grso} whose general setup applies here.
The next aim is to provide an analog of Theorem~\ref{thm:real-points-kt}
for $\mathsf{cKW}$ instead of $\mathsf{KW}$.

\begin{theorem}\label{thm:ckt-base}
  Suppose $S$ is a scheme over $\operatorname{Spec}(\mathbb{Z}[\tfrac{1}{2}])$. 
  Then $\mathsf{cKW}$ is a commutative
  $\mathbf{1}[\tfrac{1}{\eta}]$-algebra which is 
  preserved under base change.
\end{theorem}

\begin{proof}
  The statement regarding the multiplicative structure
  follows from \cite{grso}. The base change argument
  for connective covers is given in \cite[Lemma 2.2]{hoyois}.
\end{proof}

\begin{lemma}\label{lem:ckt-kt}
  The canonical map $\mathsf{cKW}\to \mathsf{KW}$ coincides with
  the canonical map
  $\mathsf{cKW} \to \mathsf{cKW}[\alpha^{-1}]$ up to canonical 
  equivalence.
\end{lemma}

\begin{proof}
  By construction, $\alpha\colon S^{8,4} \to \mathsf{KW}$ is an
  invertible element. Moreover, it lifts to
  a map $S^{8,4}\to \mathsf{cKW}$ deserving the same notation. 
  Hence there is a canonical map $\mathsf{cKW}[\alpha^{-1}]\to \mathsf{KW}$
  which induces an isomorphism on sheaves of homotopy groups,
  whence the statement.
\end{proof}

Lemma~\ref{lem:ckt-kt} implies that the filtration
on $\mathsf{KW}$ given by the homotopy $t$-structure coincides with
multiplications by powers of $\alpha$ on suspensions of
$\mathsf{cKW}$. 
Thus the cone of multiplication by $\alpha$ on $\mathsf{cKW}$, 
henceforth
denoted $\mathsf{cKW}/\alpha$, coincides with the $0$-truncation
of $\mathsf{KW}$, or equivalently, with 
the motivic Eilenberg-MacLane
spectrum for the sheaf of unramified Witt groups.
In particular, over a field $F$ of characteristic not two,
it is a motivic spectrum whose
homotopy groups are concentrated in the zero line:
\[ \pi_{s,t}\mathsf{cKW}/\alpha \cong \begin{cases} W(F) & s=t \\ 0 & s\neq t
\end{cases}
\]
Moreover,
the canonical map $\mathsf{cKW}\to \mathsf{cKW}/\alpha$ is even 
a map of $\mathbf{1}[\tfrac{1}{\eta}]$-algebras.
The homotopy $t$-structure 
filtration of $\mathsf{KW}$ thus consists of $\mathsf{cKW}$-modules,
and its associated graded consists of (de)suspensions of $\mathsf{cKW}/\alpha$. 

\begin{lemma}\label{lem:real-points-hw}
  The canonical map $\mathbb{R}^\ast(\mathsf{cKW}/\alpha)\to H\mathbb{Z}[\tfrac{1}{2}]$
  is an equivalence of topological spectra.
\end{lemma}

\begin{proof}
  The algebraic Hopf map $\eta$ induces the degree 2 map
  on the topological sphere spectrum after taking real points.
  Morel's Theorem, see~\ref{cor:morel}, implies 
  the map $\mathbf{1}[\tfrac{1}{\eta}]\to \mathsf{cKW}$
  is 1-connective and hence induces a 1-connective map 
  \[ \mathbb{S}[\tfrac{1}{2}] = \mathbb{R}^\ast(\mathbf{1}[\tfrac{1}{\eta}]) \to \mathbb{R}^\ast \mathsf{cKW} \]
  of topological spectra. This 
  identifies $\pi_0 \mathbb{R}^\ast \mathsf{cKW}$ as 
  $\mathbb{Z}[\tfrac{1}{2}]$ (and
  $\pi_1\mathbb{R}^\ast\mathsf{cKW}$ as the trivial group). 
  Since the canonical map $\mathsf{cKW} \to \mathsf{cKW}/\alpha$
  is 4-connective, also $\pi_0 \mathbb{R}^\ast \mathsf{cKW}/\alpha \cong \mathbb{Z}[\tfrac{1}{2}]$ (and
  $\pi_1\mathbb{R}^\ast\mathsf{cKW}/\alpha =0$).
  The canonical map mentioned in the
  statement of the lemma is the map to the zeroth Postnikov section.
  It remains to show that $\pi_n\mathbb{R}^\ast \mathsf{cKW}/\alpha = 0$ for all $n\geq 2$.
  For this purpose, recall that $\mathsf{cKW}/\alpha$ is the Eilenberg-MacLane
  spectrum associated to the sheaf of unramified
  Witt groups on $\mathbf{Sm}_{\mathbb{R}}$.
  Hence its real Betti realization is determined by the real Betti
  realization of the sheaf $W$ of Witt groups, considered as a complex
  concentrated in degree zero, and the spectrum structure maps. 
  Theorem~\ref{thm:ayoub-real-betti} allows to identify the
  real Betti realization of $W$ as the complex $W(\operatorname{Spec}(\mathcal{R_\bullet}))$,
  which turns out to be the complex
  \[ \mathbb{Z} \xleftarrow{0} \mathbb{Z}\xleftarrow{\mathrm{id} } \mathbb{Z} \xleftarrow{0} \mathbb{Z} \xleftarrow{\mathrm{id}} \dotsm \]
  by Lemma~\ref{lem:witt-groups-real-betti} below.
  The result follows.
\end{proof}

The statement of Lemma~\ref{lem:real-points-hw} 
for $\mathsf{cKW}[\tfrac{1}{2}]/\alpha$ is mentioned 
below the proof of \cite[Proposition~29]{bachmann.conservative}. 
Note that the proof of Lemma~\ref{lem:real-points-hw}
supplies an integral identification of the real points of
the $S^1$-Eilenberg-MacLane spectrum for the sheaf of
Witt groups.

\begin{lemma}\label{lem:witt-groups-real-betti}
  The inclusions 
  \[ \mathbb{R}\hookrightarrow \mathbb{R}[t_1,\dotsc,t_n]\hookrightarrow \mathcal{R}_n 
  \hookrightarrow \mathbb{R}\llbracket t_1,\dotsc,t_n\rrbracket \]
  induce isomorphisms on Witt groups for every $n$.
\end{lemma}

\begin{proof}
  By construction, $\mathcal{R}_0 = \mathbb{R}$. Shrinking the cube defines
  a filtered system of intermediate
  algebras between $\mathcal{R}_n$ and the formal
  power series ring $\mathbb{R}\llbracket t_1,\dotsc,t_n\rrbracket$
  whose colimit is the ring of germs of real analytic functions at
  zero. The
  intermediate algebras are all isomorphic to $\mathcal{R}_n$. 
  Moreover, the restriction homomorphism between any two such
  defines an isomorphism on Witt groups, as one concludes 
  by a homotopy interpolating between the two cube diameters.
  Since the Witt group commutes with filtered colimits, the result follows.
\end{proof}

\begin{corollary}\label{cor:real-points-ckt}
  There is an equivalence $\mathbb{R}^\ast \mathsf{cKW}[\tfrac{1}{2}] \simeq 
  \mathsf{ko}[\tfrac{1}{2}]$ of
  topological spectra. 
  The
  unit $\mathsf{cKW}[\tfrac{1}{2}]\to \mathbb{R}_\ast\mathbb{R}^\ast \mathsf{cKW}[\tfrac{1}{2}]\simeq \mathbb{R}_\ast(\mathsf{ko}[\tfrac{1}{2}])$ 
  is an equivalence
  in $\mathbf{SH}(\mathbb{R})$.
\end{corollary}

\begin{proof}
  This follows from Theorem~\ref{thm:real-points-kt}, 
  Lemma~\ref{lem:real-points-hw}, 
  and the compatibility of $(\mathbb{R}^\ast,\mathbb{R}_\ast)$ with
  the homotopy $t$-structure.
\end{proof}

Over the complex numbers, a very explicit small cell presentation of $\mathsf{cKW}$
can be given, thanks to the following 
fantastic theorem \cite{guillou-isaksen}, \cite{andrews-miller}.
This cell presentation will not be used in the remaining sections.

\begin{theorem}[Andrews-Miller]
\label{thm:gi-am}
  Over $\mathbb{C}$, 
  the graded ring $\pi_{\ast,\ast}\mathbf{1}[\tfrac{1}{\eta}]$ is isomorphic
  to the ring $\mathbb{F}_2[\eta,\eta^{-1},\sigma,\mu_9]/\eta\sigma^2$
  where $\lvert \eta\rvert = (1,1)$, $\lvert \sigma\rvert = (7,4)$,
  and $\lvert \mu_9\rvert = (9,5)$.
\end{theorem}

A priori, the Andrews-Miller computation produces the homotopy groups of
$\mathbf{1}[\tfrac{1}{\eta}]^\wedge_2$, but since $\pi_{0,0}\mathbf{1}[\tfrac{1}{\eta}]\cong \mathbb{Z}/2$
over $\mathbb{C}$
by Morel's Theorem~\ref{thm:morel}, $\mathbf{1}[\tfrac{1}{\eta}]$ 
is already 2-complete over $\mathbb{C}$.
Theorem~\ref{thm:gi-am}
implies the following statement on the slices of $\mathbf{1}[\tfrac{1}{\eta}]$, as explained in 
\cite[Theorem 2.34]{rso.sphere}.

\begin{theorem}
  \label{theorem:slices-eta-inv-sphere}
  Suppose $S$ is a scheme over $\operatorname{Spec}(\mathbb{Z}[\tfrac{1}{2}])$. 
  There is a splitting of the zero slice of the $\eta$-inverted sphere spectrum 
  \[ 
  \mathsf{s}_0(\mathbf{1}[\tfrac{1}{\eta}])
  \cong
  \bigvee_{1\neq n\geq 0} \Sigma^{n,0}  \mathsf{M} \mathbb{Z}/2.
  \]
  such that the unit $\mathbf{1}[\tfrac{1}{\eta}] \to \mathsf{KW}$ induces
  the inclusion on every even summand.
\end{theorem}

The remainder of this section
takes place in the category of 
$\mathbf{1}[\tfrac{1}{\eta}]$-modules
over the field of complex numbers. 
Abbreviate $\oplus_n \pi_{n+k,n}$ as $\pi_k$.

\begin{lemma}\label{lem:unit-mult4}
  The unit $\mathbf{1}[\tfrac{1}{\eta}]\to \mathsf{cKW}$ induces an isomorphism
  $\pi_{4k}\mathbf{1}[\tfrac{1}{\eta}]\to \pi_{4k}\mathsf{cKW}$ for every integer $k$.
\end{lemma}

\begin{proof}
  The
  unique nontrivial element $\alpha^{k}\eta^{-4k}\in \pi_{4k,0} \mathsf{KW}$
  comes from the unique nontrivial element in the 
  $4k$th column of the $E^2=\mathsf{E}^\infty$
  page of the zeroth slice spectral sequence of 
  $\mathsf{KW}$ \cite{roendigs-oestvaer.hermitian}.
  The description of the unit map $\mathbf{1}[\tfrac{1}{\eta}]\to \mathsf{KW}$ on
  slices, Theorem~\ref{theorem:slices-eta-inv-sphere}, 
  implies that this element is  the image of 
  the unique nontrivial element in the 
  $4k$th column of the $E^2=E^\infty$ page
  of the zeroth slice spectral sequence of $\mathbf{1}[\tfrac{1}{\eta}]$.
  Another proof, which does not rely on Voevodsky's slice filtration,
  is given in \cite[Theorem 3.2]{hornbostel.nil}.
\end{proof}

Set $\mathsf{D}_1=\mathbf{1}[\tfrac{1}{\eta}]$. Then the unit $\mathbf{1}[\tfrac{1}{\eta}]=\mathsf{D}_1\to \mathsf{cKW}$ is
a 3-connective map.\footnote{True over every field of characteristic not
two, as Theorem~\ref{thm:vanishing} shows.}
Set
$\mathsf{D}_2$ to be the homotopy cofiber of
\[ \sigma\eta^{-4}\colon \Sigma^{3,0}\mathbf{1}[\tfrac{1}{\eta}]\to\mathbf{1}[\tfrac{1}{\eta}] = \mathsf{D}_1 \]
Then the canonical map
$\mathbf{1}[\tfrac{1}{\eta}] =\mathsf{D}_1\to \mathsf{D}_2$ induces an isomorphism on $\pi_{4m}$,
as a consequence of the long exact sequence of homotopy groups
and Theorem~\ref{thm:gi-am}. For the same reason, the connecting map
$\mathsf{D}_2\to \Sigma^{4,0}\mathbf{1}[\tfrac{1}{\eta}]$
induces an isomorphism on $\pi_{4m+3}$ for $m\geq 1$,
giving in particular a unique nontrivial map
$\Sigma^{7,0}\mathbf{1}[\tfrac{1}{\eta}]\to \mathsf{D}_2$.
The unit
$\mathbf{1}[\tfrac{1}{\eta}]\to \mathsf{cKW}$ factors over $\mathsf{D}_2$ as a 7-connective map
$\mathsf{D}_2\to \mathsf{cKW}$ by Lemma~\ref{lem:unit-mult4}.  
This can be continued inductively, producing a sequence of 
cellular motivic spectra
factoring the unit of $\mathsf{cKW}$ as
\[ \mathbf{1}[\tfrac{1}{\eta}] =\mathsf{D}_1\to\mathsf{D}_2\to\dotsm  \to \mathsf{D}_n \to \dotsm \to \mathsf{cKW} \]
such that for every $n$ the map $\mathsf{D}_n\to \mathsf{cKW}$ 
is $4n-1$-connective and 
$ \mathbf{1}[\tfrac{1}{\eta}]\to \mathsf{D}_n \to \mathsf{cKW}$
induce isomorphisms on $\pi_{4m}$.
For every $n\geq 1$, there is a unique nontrivial element
$\Sigma^{4n-1,0}\mathbf{1}[\tfrac{1}{\eta}] \to \mathsf{D}_n$ in $\pi_{4n-1}\mathsf{D}_n\cong \pi_{4n-1}\mathbf{1}[\tfrac{1}{\eta}]$
such that 
\[ \Sigma^{4n-1,0}\mathbf{1}[\tfrac{1}{\eta}] \to \mathsf{D}_n \to 
\mathsf{D}_{n+1} \to \Sigma^{4n,0}\mathbf{1}[\tfrac{1}{\eta}] \]
is a homotopy cofiber sequence with 
$\mathsf{D}_{n+1}\to \Sigma^{4n,0}\mathbf{1}[\tfrac{1}{\eta}]$ 
inducing an isomorphism on $\pi_{4m+3}$ whenever $m\geq n$.
Taking the colimit with respect to $n\to \infty$ produces
the desired cell presentation of $\mathsf{cKW}$ by Lemma~\ref{lem:unit-mult4}.
A cell presentation for $\mathsf{KW}$ then follows from Lemma~\ref{lem:ckt-kt}.
Rationally this cell presentation splits by \cite[Corollary 2.8]{alp}, even over
any field of characteristic not two.

\section{An Adams resolution with connective Witt theory}
\label{sec:an-adams-resolution}

The section title refers to the cosimplicial diagram 
\[ [n] \mapsto \mathsf{cKW}^{\wedge n+1} \]
determined by the $\mathbf{1}[\tfrac{1}{\eta}]$-algebra $\mathsf{cKW}$.
The starting point of this resolution of $\mathbf{1}[\tfrac{1}{\eta}]$ is the following.

\begin{theorem}[Morel]
\label{thm:morel}
If $F$ is a field, $\pi_{n,n}\mathbf{1}[\tfrac{1}{\eta}]$
is isomorphic to the Witt ring of $F$.
\end{theorem}

It translates to the following statement.

\begin{corollary}
  \label{cor:morel}
  Let $F$ be a field of characteristic not two.
  The unit $cu\colon \mathbf{1}[\tfrac{1}{\eta}] \to \mathsf{cKW}$
  is 1-connective. 
\end{corollary}

In other words,  $\underline{\pi}_{n,n}cu$ is an isomorphism
and $\underline{\pi}_{n+1,n}cu$ is surjective for every integer $n$.
Let
\begin{equation}\label{eq:fib-unit-ckt}
  \mathsf{C} \to \mathbf{1}[\tfrac{1}{\eta}] \xrightarrow{cu} \mathsf{cKW}
\end{equation}
be the fiber of $cu$. Corollary~\ref{cor:morel} implies
that it is 1-connective.

\begin{lemma}\label{lem:cto1}
  The canonical
  map 
  \[  \underline{\pi}_{p,q}\mathsf{C}\to \underline{\pi}_{p,q} \mathbf{1}[\tfrac{1}{\eta}] \]
  is an isomorphism if 
  $p-q\equiv 1,2(4)$ and surjective if $p-q\equiv 3(4)$. 
\end{lemma}

\begin{proof}
  The cofiber sequence~(\ref{eq:fib-unit-ckt}) induces a 
  long exact sequence of sheaves of homotopy groups. The result
  then follows from
  vanishing $\underline{\pi}_{p,q}\mathsf{cKW} = 0$ for $p-q$ not 
  divisible by 4, which in turn follows from the fact that
  higher Witt groups of a field are concentrated in degrees congruent
  to 0 modulo 4.
\end{proof}

Smashing the cofiber sequence~(\ref{eq:fib-unit-ckt}) with $\mathsf{cKW}$
produces the following cofiber sequence:
\begin{equation}\label{eq:fib-unit-ckt-smash-ckt}
  \mathsf{C} \wedge \mathsf{cKW} \to \mathbf{1}[\tfrac{1}{\eta}]\wedge \mathsf{cKW}= \mathsf{cKW} \xrightarrow{cu\wedge \mathsf{cKW}} \mathsf{cKW} \wedge \mathsf{cKW}
\end{equation}

\begin{lemma}\label{lem:cooptoc}
  The connecting map
  \[ \underline{\pi}_{p+1,q}\mathsf{cKW}\wedge \mathsf{cKW} \to \underline{\pi}_{p,q}\mathsf{C}\wedge \mathsf{cKW} \]
  is an isomorphism for $p-q \equiv 1,2(4)$, 
  and surjective for $p-q\equiv 0,3(4)$.
\end{lemma}

\begin{proof}
  The
  cofiber sequence~(\ref{eq:fib-unit-ckt-smash-ckt}) induces
  a long exact sequence of homotopy groups. As in the proof of 
  Lemma~\ref{lem:cto1}, the vanishing 
  $\underline{\pi}_{p,q}\mathsf{cKW} = 0$ for $p-q$ not divisible
  by 4 implies the statement for $p-q \equiv 1,2,3(4)$. Since
  $cu \wedge \mathsf{cKW}$ has the multiplication as a retraction,
  surjectivity also holds for $p-q\equiv 0(4)$.
\end{proof}

An explicit consequence of Lemma~\ref{lem:cooptoc} is that
$\underline{\pi}_{n+1,n}\mathsf{cKW} \wedge \mathsf{cKW} =0$ for all 
integers $n$.
In fact, $\mathsf{C}$ is 1-connective by Corollary~\ref{cor:morel}, 
whence
$\underline{\pi}_{n,n}\mathsf{C}\wedge \mathsf{cKW} =0$.

\begin{proposition}\label{prop:coop-1}
  For every integer $n$, there is an isomorphism
  \[ \underline{\pi}_{n+2,n}\mathsf{cKW}\wedge \mathsf{cKW}  \cong 
  \underline{\pi}_{n+1,n} \mathbf{1}[\tfrac{1}{\eta}]\]
  of sheaves of homotopy groups. 
\end{proposition}

\begin{proof}
  Smashing the cofiber sequence~(\ref{eq:fib-unit-ckt}) with $\mathsf{C}$
  produces the following cofiber sequence:
  \[  \mathsf{C} \wedge \mathsf{C} \to \mathbf{1}[\tfrac{1}{\eta}]\wedge \mathsf{C}= \mathsf{C} \xrightarrow{cu\wedge \mathsf{C}} \mathsf{cKW} \wedge \mathsf{C}
  \]
  Since $\mathsf{C}$ is $1$-connective by Morel's Theorem~\ref{cor:morel},
  $\mathsf{C}\wedge \mathsf{C}$ is 2-connective, which implies
  that $\underline{\pi}_{n+1,n}\mathsf{C}\to \underline{\pi}_{n+1,n}\mathsf{cKW}\wedge \mathsf{C}$ is an isomorphism. Lemma~\ref{lem:cooptoc} gives that
  the connecting map
  $\underline{\pi}_{n+2,n}\mathsf{cKW}\wedge \mathsf{cKW} \to \underline{\pi}_{n+1,n}\mathsf{cKW}\wedge \mathsf{C}$ is an isomorphism.
  The map $\underline{\pi}_{n+1,n}\mathsf{C}\to \underline{\pi}_{n+1,n}\mathbf{1}$ is an isomorphism by Lemma~\ref{lem:cto1}. The appropriate
  composition provides the desired isomorphism.
\end{proof}

\begin{proposition}\label{prop:1implies2}
  If $\underline{\pi}_{n+1,n} \mathbf{1}[\tfrac{1}{\eta}]=0$,
  then there is an isomorphism
  \[ \underline{\pi}_{n+3,n}\mathsf{cKW}\wedge \mathsf{cKW}  \cong 
  \underline{\pi}_{n+2,n} \mathbf{1}[\tfrac{1}{\eta}]\]
  of sheaves of homotopy groups for every integer $n$.
\end{proposition}

\begin{proof}
  Assume that $\underline{\pi}_{n+1,n} \mathbf{1}[\tfrac{1}{\eta}]=0$, or,
  equivalently, 
  that $\mathsf{C}$ is 2-connective. Then
  $\mathsf{C}\wedge \mathsf{C}$ is 4-connective, which implies that 
  $\underline{\pi}_{n+k,n}\mathsf{C}\to \underline{\pi}_{n+k,n}\mathsf{cKW}\wedge \mathsf{C}$ is an isomorphism for all integers $n$ and all integers
  $k\leq 3$, and in particular for $k=2$. 
  Lemma~\ref{lem:cooptoc} implies that the connecting map
  $\underline{\pi}_{n+3,n}\mathsf{cKW}\wedge \mathsf{cKW} \to \underline{\pi}_{n+2,n}\mathsf{cKW}\wedge \mathsf{C}$ is an isomorphism. The canonical
  map $\underline{\pi}_{n+2,n}\mathsf{C}\to \underline{\pi}_{n+2,n}\mathbf{1}$ is an isomorphism by Lemma~\ref{lem:cto1}. The appropriate
  composition provides the desired isomorphism.
\end{proof}

Propositions~\ref{prop:coop-1} and~\ref{prop:1implies2} are
direct manifestations of the applicability of cooperations in connective
Witt theory to computations of motivic stable homotopy groups
of the $\eta$-inverted sphere. The following structural result  
has consequences for $\mathsf{cKW}\wedge \mathsf{cKW}$.
In order to state it, recall the real \'etale topology
on $\mathbf{Sm}_S$ which has as stalks henselian local rings
with real closed residue fields. In particular, it is finer than
the Nisnevich topology. The identity functor on motivic (symmetric) spectra
over $S$ thus can be regarded
as a left Quillen functor from the Nisnevich to the real
\'etale homotopy theory, and similarly for the 
respective $\mathbf{A}^1$-localizations. The real \'etale topology
is relevant because of \cite[Theorem 6.6]{ksw}, which implies
that a Nisnevich fibrant model for
$\mathsf{cKW}[\tfrac{1}{2}]$ is already real \'etale fibrant.
Besides being crucial input for a proof of Theorem~\ref{thm:real-points-kt}
above, this gives the base case for the following statement
(which could be formulated in greater generality).

\begin{proposition}\label{prop:ckw-modules-real-etale}
  Let $F$ be a field of characteristic zero.
  A Nisnevich fibrant $\mathsf{cKW}[\tfrac{1}{2}]$-module
  over $F$ is already real \'etale fibrant.
\end{proposition}

\begin{proof}
  The standard model structure on modules
  over a ring object has underlying weak equivalences
  and fibrations. The corresponding model structure in
  the case at hand is
  cofibrantly generated. The cofibers of these
  generating cofibrations are 
  free modules 
  $\mathsf{cKW}[\tfrac{1}{2}] \wedge \mathsf{G}$,
  where $\mathsf{G}$ is a shifted motivic suspension spectrum
  of a smooth $F$-scheme. A standard argument shows that it suffices
  to prove that a Nisnevich fibrant replacement of 
  $\mathsf{cKW}[\tfrac{1}{2}] \wedge \mathsf{G}$ is already
  real \'etale fibrant. Since the motivic symmetric spectra
  $\mathsf{G}$ are strongly dualizable, a Nisnevich fibrant
  replacement is given by the internal motivic spectrum of
  maps from (a cofibrant model of) the dual of $\mathsf{G}$
  to a Nisnevich fibrant model of
  $\mathsf{cKW}[\tfrac{1}{2}]$. However, since the latter
  is already real \'etale fibrant, so is the internal
  motivic spectrum of maps.
\end{proof}

Proposition~\ref{prop:ckw-modules-real-etale}
assumes that 2 is invertible and that the base field
has characteristic zero. Fields of positive characteristic
do not have interesting Witt theory after inverting 2,
as the following short section discusses for the sake
of completeness.
 
\section{Fields of odd characteristic}
\label{sec:fields-odd-char}

\begin{theorem}\label{thm:eta-sphere-odd-char}
  Let $F$ be a field of odd characteristic.
  Then $\mathbf{1}[\tfrac{1}{\eta},\tfrac{1}{2}]$ is contractible.
\end{theorem}

\begin{proof}
  One argument (there are simpler ones)
  uses the beginning of the $\mathsf{cKW}$-resolution of 
  $\mathbf{1}[\tfrac{1}{\eta}]$ sketched in Section~\ref{sec:an-adams-resolution}.
  The motivic spectrum $\mathsf{cKW}[\tfrac{1}{2}]$ has trivial
  homotopy sheaves by \cite[Theorem 2.6.4]{scharlau}. 
  Hence it is contractible, and so is the motivic spectrum
  $\mathsf{C}\wedge \mathsf{cKW}[\tfrac{1}{2}]$, where $\mathsf{C}$ is the homotopy fiber
  of $\mathbf{1}[\tfrac{1}{\eta}] \to \mathsf{cKW}$ described in Section~\ref{sec:an-adams-resolution}.
  It follows that the canonical map
  $\mathsf{C}\wedge \mathsf{C}[\tfrac{1}{2}] \to \mathsf{C}[\tfrac{1}{2}]$
  is a weak equivalence. However, since $\mathsf{C}$ is 1-connective by
  Morel's Theorem~\ref{cor:morel}, 
  $\mathsf{C}\wedge \mathsf{C}$ is 2-connective.
  Hence $\mathsf{C}[\tfrac{1}{2}]$ is 2-connective, so 
  $\mathsf{C}\wedge \mathsf{C}[\tfrac{1}{2}]$ is 4-connective, and so on. 
  Thus $\mathsf{C}[\frac{1}{2}]$ is contractible, which implies the
  same for $\mathbf{1}[\tfrac{1}{\eta},\tfrac{1}{2}]$.
\end{proof}

Theorem~\ref{thm:eta-sphere-odd-char} holds
for any field whose Witt group is 
(necessarily 2-primary) torsion. This class of fields
coincides with the class of non formally real 
fields \cite[Theorem 2.7.1]{scharlau}. 

\section{Witt cooperations} 
\label{sec:oper-high-witt}

All statements on Witt cooperations will
be deduced from the rational case \cite[Theorem 10.2]{ananyevskiy.coop}
and the following topological result.

\begin{theorem}[Mahowald, Kane, Lellmann]\label{thm:real-coop}
  For every prime $p$, there exists a sequence $B(j)$ of connective 
  topological spectra starting with $S^0$, 
  a sequence $s_j$ 
  of natural numbers with $s_j\geq 4j$, 
  and a $p$-local equivalence
  \begin{equation}\label{eq:real-coop} 
    \gamma\colon \bigvee_{j\geq 0} 
    \Sigma^{s_j}B(j)\wedge \mathsf{ko} \to \mathsf{ko}\wedge \mathsf{ko} 
  \end{equation}
  of topological spectra.
\end{theorem}

\begin{proof}
  This follows from \cite{mahowald.bo} for the prime $2$ (which is
  irrelevant for the following arguments), \cite{kane},
  and \cite{lellmann} for odd primes. 
  See also \cite{milgram.bo} and \cite{davis.bo}.
\end{proof}

Theorem~\ref{thm:real-coop} can be reinterpreted as follows.
Fix an odd prime. Consider the motivic spectrum 
$\Sigma^{s_j,0}B(j)\wedge \mathsf{cKW}[\tfrac{1}{2}]$
over the real numbers. Here the smash product with the topological
spectrum $B(j)$ can be interpreted in two equivalent ways. One way is to view 
the category of
motivic (symmetric) spectra as enriched over the category of
usual (symmetric) spectra, and to use an appropriate simplicial version of
$B(j)$. Another way, to be pursued here, 
is to consider $B(j)$ as a constant presheaf of usual
$S^1$-spectra, having a motivic suspension spectrum with the same notation. 
The real Betti realization of 
$\Sigma^{s_j,0}B(j)\wedge \mathsf{cKW}[\tfrac{1}{2}]$ is
\[ \mathbb{R}^\ast\bigl(\Sigma^{s_j,0}B(j)\wedge \mathsf{cKW}[\tfrac{1}{2}]\bigr)
\simeq
\mathbb{R}^\ast\Sigma^{s_j,0}B(j)\wedge \mathbb{R}^\ast\mathsf{cKW}[\tfrac{1}{2}])\simeq \Sigma^{s_j}B(j)\wedge \mathsf{ko}[\tfrac{1}{2}]\]
by Theorem~\ref{thm:real-points} and Corollary~\ref{cor:real-points-ckt}.
Composition with $\gamma$ from~(\ref{eq:real-coop}) yields
a map 
\begin{equation}\label{eq:2} 
\mathbb{R}^\ast\bigl(\Sigma^{s_j,0}B(j)\wedge \mathsf{cKW}[\tfrac{1}{2}]\bigr)
\to 
\mathsf{ko}\wedge \mathsf{ko}[\tfrac{1}{2}] 
\xrightarrow{\simeq} 
\mathbb{R}^\ast\bigl(\mathsf{cKW} \wedge \mathsf{cKW}[\tfrac{1}{2}]\bigr)
\end{equation}
whose adjoint is a map
$\Sigma^{s_j,0}B(j)\wedge \mathsf{cKW}[\tfrac{1}{2}] \to \mathbb{R}_\ast\mathbb{R}^\ast\bigl (\mathsf{cKW}\wedge \mathsf{cKW}[\tfrac{1}{2}]\bigr).$
Let 
\[
\gamma^\flat\colon
\Sigma^{s_j,0}B(j)\wedge \mathsf{cKW}[\tfrac{1}{2}] \to \mathbb{R}_\ast\mathbb{R}^\ast\bigl (\mathsf{cKW}\wedge \mathsf{cKW}[\tfrac{1}{2}]\bigr)
\simeq 
\mathsf{cKW}\wedge \mathbb{R}_\ast \mathbb{R}^\ast\bigl(\mathsf{cKW}[\tfrac{1}{2}]\bigr)
\simeq
\mathsf{cKW}\wedge \mathsf{cKW}[\tfrac{1}{2}]
\]
be the composition of this adjoint map, the
equivalence occurring in the projection formula~(\ref{eq:real-points})
for the real Betti realization displayed in Theorem~\ref{thm:real-points}, 
and finally the equivalence
$\mathsf{cKW}[\tfrac{1}{2}]\to \mathbb{R}_\ast\mathbb{R}^\ast(\mathsf{cKW}[\tfrac{1}{2}])$ from
Corollary~\ref{cor:real-points-ckt}.

\begin{corollary}\label{cor:ckt-coop}
  Let $p$ be an odd prime. The map 
  \[ \gamma^\flat\colon \bigvee_{j\geq 0} \Sigma^{s_j,0}B(j)\wedge \mathsf{cKW}[\tfrac{1}{2}] \to \mathsf{cKW} \wedge \mathsf{cKW}[\tfrac{1}{2}] \]
  induced by the 
  $p$-local equivalence~(\ref{eq:real-coop})
  is a $p$-local 
  equivalence in 
  $\mathbf{SH}(\mathbb{R})$.
\end{corollary}

\begin{proof}
  The adjunction $(\mathbb{R}^\ast,\mathbb{R}_\ast)$ of integral stable
  homotopy categories descends to the ``same'' adjunction of
  $p$-local stable homotopy categories. 
  Theorem~\ref{thm:real-coop} implies that the map~(\ref{eq:2})
  is a $p$-local equivalence. Its adjoint 
  $\Sigma^{s_j,0}B(j)\wedge \mathsf{cKW}[\tfrac{1}{2}] \to \mathbb{R}_\ast\mathbb{R}^\ast\bigl (\mathsf{cKW}\wedge \mathsf{cKW}[\tfrac{1}{2}]\bigr)$
  is computed as the image of~(\ref{eq:2}) under $\mathbb{R}_\ast$,
  composed with the unit of the adjunction. Since $\mathbb{R}_\ast$
  preserves $p$-local equivalences, it suffices to show
  that the unit
  \[ \Sigma^{s_j,0}B(j)\wedge \mathsf{cKW}[\tfrac{1}{2}] \to 
  \mathbb{R}_\ast \mathbb{R}^\ast \Sigma^{s_j,0}B(j)\wedge \mathsf{cKW}[\tfrac{1}{2}] \]
  is a $p$-local equivalence. However, it is in fact an equivalence,
  by Theorem~\ref{cor:real-points-ckt} and the projection formula from
  Theorem~\ref{thm:real-points}.
\end{proof}

Corollary~\ref{cor:ckt-coop} implies the same 
$p$-local equivalence
over any real closed field. 

\begin{theorem}\label{thm:ckt-coop-conc}
  Let $F$ be a field of characteristic zero.
  Then the map 
  \[ cu\wedge \mathsf{cKW}[\tfrac{1}{2}]\colon \mathsf{cKW}[\tfrac{1}{2}] \to \mathsf{cKW} \wedge \mathsf{cKW}[\tfrac{1}{2}]\]
  is 3-connective. In particular, $\underline{\pi}_{s,t}\mathsf{cKW}\wedge \mathsf{cKW}[\tfrac{1}{2}]$
  is trivial for $1\leq s-t\leq 3$.
\end{theorem}

\begin{proof}
  The source $\mathsf{cKW}[\tfrac{1}{2}]$ 
  of the map $cu\wedge \mathsf{cKW}[\tfrac{1}{2}]$
  satisfies real \'etale descent by \cite[Theorem 6.6]{ksw},
  and its target does so by Proposition~\ref{prop:ckw-modules-real-etale}.
  The statement follows by proving that the induced map
  \begin{equation}\label{eq:homotopy-sheaves}
    \underline{\pi}_{s,t} \mathsf{cKW}[\tfrac{1}{2}] \to \underline{\pi}_{s,t} \mathsf{cKW} \wedge \mathsf{cKW}[\tfrac{1}{2}] 
  \end{equation}
  on homotopy sheaves for the real \'etale topology
  is an isomorphism for all $s-t\leq 3$.
  Hence
  it suffices to prove that (\ref{eq:homotopy-sheaves})
  is an isomorphism for all $s-t\leq 3$ after evaluating
  at a henselian local ring with real closed residue field.
  The rigidity property of Witt groups \cite[Satz 3.3]{knebusch}
  provides that it
  suffices to evaluate (\ref{eq:homotopy-sheaves}) at
  real closed fields. 

  The main results 
  of \cite{alp} imply that $cu\wedge \mathsf{cKW}_{\mathbb{Q}}\colon \mathsf{cKW}_\mathbb{Q} \to 
  \mathsf{cKW} \wedge \mathsf{cKW}_\mathbb{Q}$
  is 3-connective. More precisely,
  \cite[Corollary 3.5]{alp} and \cite[Theorem 3.4]{alp}
  imply that 
  $\mathsf{cKW}_{\mathbb{Q}}\simeq \bigvee_{n\geq 0} \Sigma^{8n,4n}
  \mathbf{1}[\tfrac{1}{\eta}]_{\mathbb{Q}}$
  whence 
  \[ \mathsf{cKW}_{\mathbb{Q}}\wedge \mathsf{cKW}_{\mathbb{Q}} 
  \simeq \bigvee_{n\geq 0} \Sigma^{8n,4n} \vee_{j=0}^n 
  \mathbf{1}[\tfrac{1}{\eta}]_{\mathbb{Q}}. \]
  In particular,
  $\underline{\pi}_{s,t}\mathsf{cKW}\wedge \mathsf{cKW}_\mathbb{Q}$
  is trivial for $1\leq s-t\leq 3$. Hence it suffices 
  to show that $p$-torsion vanishes in that range one odd prime 
  $p$ at a time. But this is an immediate consequence of 
  Corollary~\ref{cor:ckt-coop}
  which holds for real closed fields.
\end{proof}

More precise information on Witt cooperations, at
least after inverting 2, can be deduced from
knowledge about the topological spectra $B(j)$.
They are given as Thom spectra of maps from even parts 
of the May filtration on the homotopy fiber of 
$\Omega^2 S^3\to S^1$ to $\mathsf{KO}$;
see \cite[Section 2]{mahowald.bo} for details.

\section{The vanishing}
\label{sec:vanishing}

As mentioned already in Section~\ref{section:introduction},
considering connective Witt theory with 2 inverted suffices
due to knowledge on $(2,\eta)$-completed groups from~\cite{rso.sphere}, 
which coincide with $2$-completed groups under the following
circumstances.

\begin{theorem}[Hu-Kriz-Ormsby]
\label{thm:hu-kriz-ormsby}
  Let $F$ be a field of finite virtual 2-cohomological
  dimension and of characteristic not two. Then
  the canonical map $\mathbf{1}^\wedge_2\to \mathbf{1}^\wedge_{2,\eta}$
  is an equivalence.
\end{theorem}

\begin{proof}
  The reference~\cite[Theorem 1]{hu-kriz-ormsby} provides this statement
  in the case of fields of characteristic zero. Their proof
  goes through for fields of odd characteristic
  with the amendment that the motivic Eilenberg-MacLane
  spectrum $\mathsf{M} \mathbb{Z}/2$ admits a cell presentation of finite type
  also over fields of characteristic not 
two \cite[Proposition 3.32]{rso.sphere}.
\end{proof}

\begin{lemma}\label{lem:unit-inv-2}
  Let $F$ a field of characteristic
  not two. The canonical map
  \[ \pi_{n+1,n}\mathbf{1} \to \pi_{n+1,n}\mathbf{1}[\tfrac{1}{2}]\]
  is injective for every natural number $n\geq 5$.
\end{lemma}

\begin{proof}
  Since the base field is a filtered colimit
  of fields of finite virtual 2-cohomological dimension,
  and $\pi_{n+1,n}\mathbf{1}$ commutes with filtered colimits
  of fields, one may assume that the base field has
  finite virtual 2-cohomological dimension.
  Consider the arithmetic square for $2$ and $\mathbf{1}$.
  It induces a long exact sequence 
  \[ \dotsm \to \pi_{n+2,n}\mathbf{1}^\wedge_2[\tfrac{1}{2}]
  \to \pi_{n+1,n}\mathbf{1} \to \pi_{n+1,n}\mathbf{1}[\tfrac{1}{2}] \oplus \pi_{n+1,n}\mathbf{1}^\wedge_2 
  \to \pi_{n+1,n}\mathbf{1}^\wedge_2[\tfrac{1}{2}] \to \dotsm \]
  of homotopy groups. The argument below will provide 
  that $\pi_{n+2,n}\mathbf{1}^\wedge_2=0$. Theorem~\ref{thm:hu-kriz-ormsby}
  implies that the natural map
  $\pi_{s,t}\mathbf{1}^\wedge_2\to \pi_{s,t}\mathbf{1}^\wedge_{2,\eta}$
  is an isomorphism for all $s,t$. The
  convergence result \cite[Theorem 3.50]{rso.sphere}
  supplies a weak equivalence
  $\mathsf{sc}(\mathbf{1})\simeq \mathbf{1}^{\wedge}_{\eta}$ between the
  $\eta$-completion and the completion with respect to
  the slice filtration of $\mathbf{1}$, hence
  an isomorphism $\pi_{s,t}\mathbf{1}^\wedge_{2,\eta} \cong 
  \pi_{s,t}(\mathsf{sc}\mathbf{1})^\wedge_{2}$ for all $s,t$.
  Consider Milnor's derived limit exact sequence:
  \begin{equation}\label{eq:milnor-seq} 
    0 \to {\lim_{k}}^1 \pi_{s+1,t} (\mathsf{sc}\mathbf{1})/2^k 
    \to \pi_{s,t}(\mathsf{sc}\mathbf{1})^\wedge_{2} 
    \to \lim_k\pi_{s,t} (\mathsf{sc}\mathbf{1})/2^k \to 0 
  \end{equation}
  The canonical map $(\mathsf{sc}\mathbf{1})/2^k \to \mathsf{sc}(\mathbf{1}/2^k)$
  is an equivalence. As stated in the beginning of~\cite[Section 5]{rso.sphere},
  the computation of the first slice differential for the motivic sphere
  spectrum~\cite[Lemma 4.1]{rso.sphere} implies that
  $\pi_{n+1,n}\mathsf{sc}(\mathbf{1}) = 0$ for $n\geq 3$ and
  $\pi_{n+2,n}\mathsf{sc}(\mathbf{1}) = 0$ for $n\geq 5$. The long exact
  sequence of homotopy groups hence forces 
  \[ \pi_{n+2,n}\mathsf{sc}(\mathbf{1})/2^k = 0 \quad  \mathrm{for\ }
  n\geq 5 \mathrm{\ and\ }  k\geq 1 \]
  which -- together with Milnor's short exact 
  sequence~(\ref{eq:milnor-seq}) --  implies the isomorphisms
  \[ {\lim_{k}}^1 \pi_{n+3,n} (\mathsf{sc}\mathbf{1})/2^k 
  \xrightarrow{\cong} \pi_{n+2,n}(\mathsf{sc}\mathbf{1})^\wedge_{2} 
  \quad \mathrm{and}
  \quad \pi_{n+1,n}(\mathsf{sc}\mathbf{1})^\wedge_{2} 
  \xrightarrow{\cong} \lim_k\pi_{n+1,n} (\mathsf{sc}\mathbf{1})/2^k. \]
  It also implies that the canonical map 
  $\pi_{n+3,n}(\mathsf{sc}\mathbf{1})\to
  \pi_{n+3,n}(\mathsf{sc}\mathbf{1})/2^k$
  is surjective for all $n\geq 5$ and all $k\geq 1$. 
  Hence the canonical map 
  $\pi_{n+3,n}(\mathsf{sc}\mathbf{1})/2^{k+1}   
  \to  \pi_{n+3,n}(\mathsf{sc}\mathbf{1})/2^{k}$
  is surjective for all $k\geq 1$ by the commutative diagram
  \[ \xymatrix{
    \pi_{n+3,n}(\mathsf{sc}\mathbf{1})\ar[r] \ar[d]_-{\mathrm{id}} &
    \pi_{n+3,n}(\mathsf{sc}\mathbf{1})/2^{k+1} \ar[d] \\
    \pi_{n+3,n}(\mathsf{sc}\mathbf{1})\ar[r]  &
    \pi_{n+3,n}(\mathsf{sc}\mathbf{1})/2^{k}.}\]
  In particular,  
  \[ 0={\lim_{k}}^1 \pi_{n+3,n} (\mathsf{sc}\mathbf{1})/2^k 
  = \pi_{n+2,n}(\mathsf{sc}\mathbf{1})^\wedge_{2} = 
  \pi_{n+2,n}\mathbf{1}^\wedge_{2}\] for all $n\geq 5$.
  Thus the map 
  \[ \pi_{n+1,n}\mathbf{1} \to \pi_{n+1,n}\mathbf{1}[\tfrac{1}{2}] \oplus \pi_{n+1,n}\mathbf{1}^\wedge_2 \]
  is injective for all $n\geq 5$. It remains to observe that the
  map $\pi_{n+1,n}\mathbf{1} \to  \pi_{n+1,n}\mathbf{1}^\wedge_2$ is
  the zero map for all $n\geq 5$. In fact, the isomorphisms
  $\pi_{n+1,n}\mathbf{1}^\wedge_2 \cong \pi_{n+1,n}\mathsf{sc}(\mathbf{1})^\wedge_2
  \cong \lim_k\pi_{n+1,n} (\mathsf{sc}\mathbf{1})/2^k$ explained above
  show that it suffices to prove the triviality of the canonical
  map $\pi_{n+1,n}\mathbf{1} \to  \pi_{n+1,n}\mathsf{sc}(\mathbf{1})/2^k$
  for all $n\geq 5$ and all $k\geq 1$, which again follows from the
  vanishing of $\pi_{n+2,n}\mathsf{sc}(\mathbf{1})$ for all $n\geq 5$.
\end{proof}

The statement in Lemma~\ref{lem:unit-inv-2} is not optimal. The 
computation \cite[Theorem 5.5]{rso.sphere} which is based on
Theorem~\ref{thm:vanishing} below implies that $\pi_{n+1,n}\mathbf{1}=0$
for $n\geq 3$. The map 
$ \pi_{3,2}\mathbf{1} \to \pi_{3,2}\mathbf{1}[\tfrac{1}{2}]$
is not injective; up to isomorphism it is the projection
$\mathbb{Z}/24\to \mathbb{Z}/3$ by \cite[Theorem 5.5]{rso.sphere}.

\begin{theorem}\label{thm:vanishing}
  Let $F$ be a field of characteristic not two. Then
  $\underline{\pi}_{n+1,n}\mathbf{1}[\tfrac{1}{\eta}] = \underline{\pi}_{n+2,n}\mathbf{1}[\tfrac{1}{\eta}] = 0$ for every integer $n$.
\end{theorem}

\begin{proof}
  Lemma~\ref{lem:unit-inv-2} (which holds on the level
  of homotopy sheaves, for example by \cite[Theorem 2.7]{hoyois}) shows that it suffices to
  prove the statement for $\mathbf{1}[\tfrac{1}{\eta},\tfrac{1}{2}]$.
  Thanks to Theorem~\ref{thm:eta-sphere-odd-char},
  it is enough to consider fields of characteristic zero.
  Proposition~\ref{prop:coop-1} implies the isomorphism
  \[ \underline{\pi}_{n+1,n} \mathbf{1}[\tfrac{1}{\eta},\tfrac{1}{2}] \cong \underline{\pi}_{n+2,n}\mathsf{cKW}\wedge \mathsf{cKW}[\tfrac{1}{2}] \]
  for every integer $n$. Hence this sheaf of homotopy groups 
  vanishes by Theorem~\ref{thm:ckt-coop-conc}.
  This in turn shows that $\mathbf{1}[\tfrac{1}{\eta}] \to \mathsf{cKW}$ 
  is $2$-connective,
  whence Proposition~\ref{prop:1implies2} implies an isomorphism
  \[ \underline{\pi}_{n+2,n} \mathbf{1}[\tfrac{1}{\eta},\tfrac{1}{2}] \cong \underline{\pi}_{n+3,n}\mathsf{cKW}\wedge \mathsf{cKW}[\tfrac{1}{2}] \]
  for every integer $n$.
  The latter sheaf again vanishes by Theorem~\ref{thm:ckt-coop-conc}.
\end{proof}

Theorem~\ref{thm:vanishing} shows that Theorem~\ref{thm:ckt-coop-conc}
already holds before inverting two. The vanishing of 
homotopy sheaves in Theorem~\ref{thm:vanishing}
implies the vanishing of the homotopy groups
\[ {\pi}_{n+1,n}\mathbf{1}[\tfrac{1}{\eta}] = {\pi}_{n+2,n}\mathbf{1}[\tfrac{1}{\eta}] = 0\] for all integers $n$ stated in Theorem~\ref{thm:introduction}.

\myaddress{Institute of Mathematics, University of Osnabr\"uck, Osnabr\"uck, Germany}
\myemail{oliver.roendigs@uni-osnabrueck.de}

\end{document}